\documentclass[a4paper, 12pt]{amsart}   

\usepackage{mathptmx, amssymb,amscd,latexsym, eulervm} 
\usepackage{amsmath}
\usepackage{amsthm}
\usepackage{mathdots}
\usepackage{setspace}
\usepackage{hyperref}
\hypersetup{
	colorlinks,
	linkcolor={blue!90!black},
	citecolor={red!80!black},
	urlcolor={blue!50!black},
	breaklinks=true
}
\usepackage{tabularx}
\usepackage{comment}
\usepackage{amsfonts}
\usepackage{paralist}
\usepackage{aliascnt}
\usepackage[initials, lite]{amsrefs}
\usepackage{amscd}
\usepackage{blkarray}
\usepackage{booktabs}
\usepackage{enumitem}
\usepackage{orcidlink}
\usepackage{setspace}
\usepackage[inner=2.5cm,outer=2.5cm, bottom=3.2cm]{geometry}
\usepackage{hyperref}
\usepackage[capitalize,nameinlink]{cleveref}
\usepackage{tikz, tikz-cd}
\usepackage{calligra,mathrsfs}
\usetikzlibrary{matrix}
\usetikzlibrary{arrows,calc}
\usepackage{url}
\usepackage{amssymb,latexsym,amsmath,amsfonts,amsthm,amscd,graphicx,url,color,hyperref,stmaryrd}

\theoremstyle{plain}

\newtheorem{thm}{Theorem}

\newtheorem{cor}[thm]{Corollary}

\newtheorem{que}[thm]{Question}

\theoremstyle{definition}

\newtheorem{remark}[thm]{Remark}

\newtheorem{example}[thm]{Example}

\newtheorem{discussion}[thm]{Discussion}

\newcommand{\mm}{\mathfrak{m}}
\newcommand{\mc}{\mathfrak{c}}

\title[Upper bound for multiplicity and Wilf Conjecture]{An upper bound for the multiplicity and Wilf's conjecture for one-dimensional Cohen-Macaulay rings}

\author{Marco D'Anna \orcidlink{0000-0002-9943-5527}}

\address[Marco D'Anna]{Dipartimento di Matematica e Informatica, \ Universit\`a di Catania, \  Viale Andrea Doria 6, 
	95125 Catania,Italy}

\email{marco.danna@unict.it}

\author{Alessio Moscariello \orcidlink{0000-0001-8050-4281}}

\address[Alessio Moscariello]{Department of Mathematics and Scientific Computing, \ University of Graz, NAWI Graz, \ Heinrichstra{\ss}e 36, 8010 Graz, Austria}

\email{alessio.moscariello@uni-graz.at}

\subjclass[2020]{13H10, 13H15, 20M14}

\keywords{multiplicity; Cohen-Macaulay ring; conductor; reduction; almost Gorenstein ring; Wilf's conjecture.}

\begin{document}
	
	\begin{abstract}
		In this work we provide an upper bound for the multiplicity of a one-dimensional Cohen-Macaulay ring (under certain conditions), describe the rings attaining the equality for this bound, and outline a connection with Wilf's conjecture for numerical semigroup rings. 
		Then we prove the analogue of Wilf's conjecture for almost Gorenstein rings.
		
		%We also discuss these assumptions in higher dimension, proving that they are never satisfied if the dimension is larger than one, and provide some examples of rings which are not Cohen-Macaulay but still satisfy our bound.
		
	\end{abstract}
	
	\maketitle
	
	\section*{Introduction}
	
	Let $(R,\mathfrak m)$ be a local Cohen-Macaulay ring.
	In a famous paper by S. S. Abyankhar (\cite{Ab}) it is proved that 
	the multiplicity of $R$
	(that is usually denoted by $e(R)$) is bounded below 
	in function of the embedding dimension (that we will denote by $\nu(R)$ or, simply, by $\nu$) and the Krull-dimension $d$ of $R$:
	$$
	e(R) \geq \nu+1-d
	$$
	This bound is sharp and rings for which the equality is attained are called rings of minimal multiplicity (or of maximal embedding dimension)
	and have particularly nice properties (see e.g. \cite{S}). On the other hand,
	upper bounds for the multiplicity involving the binomial coefficient $\binom{\nu}{d}$ are given in \cite{HW} and \cite{KZ}, under particular hypotheses -
	mainly in positive characteristic.
	
	In this paper we propose another upper bound for one-dimensional Cohen-Macaulay rings, under the assumption that the integral closure $\overline R$ of $R$ in its total ring of fractions is a finite $R$-module. This bound is  dependent on $\nu$ and on the conductor ideal $\mc = R:\overline R$; more precisely, in Theorem \ref{dimd} we prove that:
	$$e(R) \le (\nu-1)\ell(R/\mc) + 1.$$
	
	This bound has been inspired by a result we obtained in a previous paper
	(\cite{DM}) about numerical semigroups and, consequently, about monomial curves (or numerical semigroup rings).
	However, even working in dimension one, its generalization to Cohen-Macaulay local rings is not straightforward and requires 
	ring theory techniques, and, as noticed above, the assumption of appropriate hypotheses.
	
	In the context of numerical semigroups, our bound is also related to a long-standing conjecture, due to Wilf (cf. \cite{Wilf}), which has been studied for decades (see \cite{D} for the state of the art on this conjecture). 
	More precisely, Theorem \ref{dimd}, combined with Lech's inequality (\cite{L}[Theorem 3]), implies that the Hilbert-Samuel multiplicty of $\mc$
	can be bounded as follows: $e(\mc) \le \nu\cdot\ell(R/\mc)^2$ (see Corollary \ref{pseudowilfring}).
	On the other hand, Wilf's conjecture states that $e(\mc) \le \nu \cdot \ell(R/\mc)$  for numerical semigroup rings.
	Thus one could ask whether it is possible to factor out $\ell(R/\mc)$ from the inequality of the above mentioned corollary
	(see Question \ref{pseudoWilfConj}) even in this more general context. This idea is confirmed by the fact that the bound given in Corollary \ref{pseudowilfring} is sharp if and only if $\ell(R/\mc)=1$. 
	
	In Section 2, we restrict our attention to almost Gorenstein rings. In recent years several authors introduced new classes of Cohen-Macaulay rings, which are "close" to Gorenstein rings in certain respects. Among these classes, one of the most relevant if that of almost Gorenstein rings, introduced for the one-dimensional analytically unramified case in \cite{BF} and recently generalized in \cite{GM} and then in \cite{GTT}. This class of rings has a numerical analogue in the class of almost symmetric numerical semigroups (see again \cite{BF}). It has been proved in \cite{B} and in \cite{DM} that almost symmetric semigroups satisfy Wilf's conjecture, but both the proofs are based on numerical arguments. Here we give a proof that holds for one-dimensional rings satisfying the hypotheses of this paper (see Theorem \ref{AGWilf}). This proof can be extended to rings with canonical reduction (see \cite{R}), provided that there exists a minimal generator of the maximal ideal, different by the minimal reduction, that belongs to $\mm:\omega_R$.

	%After proving, in Section 1, the main results (Theorem \ref{dimd} and Corollary \ref{pseudowilfring}), we explore the connection with numerical semigroup theory and Wilf's conjecture. Then, in Section 2, we discuss the assumptions needed to prove the main results, and then we show that these assumptions are never satisfied for rings with dimension at least $2$. This in particular implies that certain classes of rings with dimension $d > 1$ are not Cohen-Macaulay (Examples \ref{ex1}-\ref{ex3}). However, our bounds still hold for these classes of rings, thus making it sensible to ask whether a further generalization is possible. Since the techniques used rely heavily on the Cohen-Macaulay hypothesis, an improvement of this result beyond the class of one-dimensional Cohen-Macaulay rings might require different ideas.
	
	\section{Bound for the multiplicity}
	
	Let $(R,\mm)$ be a one-dimensional, Cohen-Macaulay local ring, with infinite residue field $R/\mm$. 
	Assume that $R$ is analytically unramified, i.e., by definition, that its $\mm$-adic completion $\hat R$ is a reduced ring. This last condition is equivalent to the fact that the integral closure $\overline R$ of $R$ in its total ring of fractions $Q(R)$ is a finite $R$-module
	(cf. \cite[Theorem 10.2]{M}).
	
	From our assumptions, it follows that $R$ is a reduced ring (since $\hat R$ is reduced), and that the conductor $\mc= R:\overline R$ is a non-zero ideal of $R$ containing a non-zero divisor;
	therefore $\mc$ is an $\mm$-primary ideal and the length 
	$\ell(R/\mc)$ is finite.
	
	Moreover, by the Cohen-Macaulay property, $\mm$ contains a non-zero divisor; this implies that any minimal reduction $(x)$ of $\mm$ must be generated by a non-zero divisor, which can be chosen as an element of a minimal set of generators of $\mm$. By the Cohen-Macaulayness of $R$ it also follows that its multiplicity $e(R)$ can be computed as the length $\ell(R/(x))$ (see \cite{HS}[Proposition 11.2.2]). From this fact we also immediately get that $e(R)=\ell(J/xJ)$, for any $\mm$-primary ideal $J$; in fact, $\ell(R/xJ)=\ell(R/(x))+\ell((x)/xJ)=\ell(R/J)+\ell(J/xJ)$ and, since $x$ is a non-zerodivisor, $\ell(R/J)=\ell((x)/xJ)$, and therefore $\ell(R/(x))=\ell(J/xJ)$.
	
	\begin{remark}\label{m}
		If $(x)$ is a minimal reduction of $\mm$, we have that $\mm=x\overline R \cap R$. In fact, denoting by $P_1, \dots , P_h$ the minimal primes of $R$, $R  \hookrightarrow R/P_1 \times \dots \times  R/P_h$, $Q(R) \cong Q(R/P_1) \times \dots \times Q(R/P_h)$ and $\overline R \cong \overline{(R/P_1)} \times \dots \times \overline{(R/P_h)}$, where $\overline{(R/P_i)}$ is the integral closure of $R/P_i$ in its quotient field $Q(R/P_i)$ (see \cite[Proposition 5.17]{G}). Now let $M$ be any maximal ideal of $\overline R$; then $\overline R_M$ is a DVR, hence $(x\overline R)_M=x\overline R_M$ is an integrally closed ideal, and, by \cite[Proposition 1.1.4]{HS}, it follows that $x\overline R$ is also integrally closed.
		Therefore, by \cite[Proposition 1.6.1]{HS}, the integral closure of $(x)$ as ideal of $R$ equals $x\overline R \cap R$, and, being $(x)$ a reduction of $\mm$, we get $\mm \subseteq \overline{(x)}=x\overline R \cap R$, that, in turn, implies the desired equality.
		
		%Notice that the same argument can be used to prove that every ideal of $\overline R$ is integrally closed.
		%In particular, $\mc$ is integrally closed in $\overline R$ and, therefore, it is integrally closed also in $R$.
		
	\end{remark}
	
	\begin{thm}\label{dimd}
		Let $(R,\mm)$ be a one-dimensional, Cohen-Macaulay, analytically unramified, local ring, with embedding dimension $\nu$, 
		such that the residue field $R/\mm=\Bbbk$ 
		is infinite, and let $(x)$ be a minimal reduction of $\mm$. Then $$e(R) \le (\nu-1)\ell(R/\mc) + 1.$$
		Moreover, the equality occurs if and only if either $R$ is a DVR or $\ell(R/\mc)=1$.
	\end{thm}
	
	\begin{proof}
		Being $\mc$ an ideal of $\overline R$, we have	$x\mc = x\overline R\mc = \mm\overline R\mc = \mm\mc$. Therefore the number of generators of $\mc$, which
		is the vector space dimension of $\mc/\mm\mc$, is also the length of $\mc/x\mc$, which is the
		multiplicity of $R$.
		Hence we have $\ell(R/\mc) + e(R) = \ell(R/\mc) + \ell(\mc/\mm\mc) = \ell(R/\mm\mc) = 1 + \ell(\mm/\mm\mc)$.
		
		Now $(R/\mc)^{\nu}$ maps onto $\mm/\mm\mc$, so  we get the  inequality
		$\ell(R/\mc) + e(R) \leq 1 + \nu\cdot
		\ell(R/\mc)$
		which is equivalent to inequality in the statement.
		
		It is clear that if $R$ is a DVR then $\nu=1$, $e(R)=1$ and thus the inequality is trivial. On the other hand, if $R$ is not a DVR (and thus is not regular) then equality occurs if and only if the map from $(R/\mc)^{\nu}$
		to $\mm/\mm\mc$ is an isomorphism, that is if and only if $\mm/\mm\mc$ is a free $R/\mc$ module. In particular, since the Koszul relations on the
		generators of the maximal ideal are always in the kernel of that map, equality can occur if and only if $\mm=\mc$.
	\end{proof}
	
	%	The bound obtained in the previous theorem is not optimal. This fact is apparent by looking at the proof: in the last inequalities the invariant $\ell(R/\mc)$ comes into play as an upper bound for the cardinality $|\Omega_j|$, since we proved that the elements of these sets are linearly independent. In general, the elements of the sets $\Omega_j := \{\overline{(\frac{\omega_i}{x_j})} \ | \ x_j  \text{ divides } \omega_i\} \subseteq R/\mc$ are in the subspace of $R/\mc$ spanned by elements of $(R \setminus (x_1))/\mc$; their span must somehow reflect the structure of $R/(x_1)$ as a $\Bbbk$-vector space. Our bound seems rather loose, but it is difficult to describe this structure using only these invariants and without relying on the multiplicity itself. Moreover, while our bound is loose in general, it is also sharp: Proposition \ref{sharp} shows that equality is attained if and only if $\ell(R/\mc)=1$.
In the next corollary we give a bound
on the Hilbert-Samuel multiplicity $e(\mc)$ of $\mc$ in function of the embedding dimension of $R$. Notice that this multiplicity is defined when the ideal $\mc$ is $\mm$-primary. Therefore, compared to the previous result, in the next result $R$ cannot be a DVR, since this would imply that $\mc=R$.
	\begin{cor}\label{pseudowilfring}
		Let $(R,\mm)$ be a one-dimensional, Cohen-Macaulay, analytically unramified, local ring, with embedding dimension $\nu$, such that the residue field $R/\mm=\Bbbk$ is infinite and the conductor $\mc$ is $\mm$-primary. Then
		 $$e(\mc) \le (\nu-1)\ell(R/\mc)^2+\ell(R/\mc).$$
		
		In particular, $e(\mc) \le \nu \cdot \ell(R/\mc)^2$. Moreover, then equality holds in both inequalities if an only if $\ell(R/\mc)=1$.
	\end{cor}
	
	\begin{proof}
		We use the following well-known Theorem by Lech, which states that, for a $\mm$-primary ideal $I$, $e(I) \le d!e(R)\ell(R/I)$ (cf. \cite{L}[Theorem 3]) - where $e(I)$ is the Hilbert-Samuel multiplicity of $I$ and $d$ denotes the Krull dimension of $R$. In our context, this gives $$e(\mc) \le e(R)\ell(R/\mc),$$
		and replacing $e(R)$ with the bound obtained in Theorem \ref{dimd} we obtain our first inequality. The second one is straightforward.
		
		From the above argument we deduce immediately that equality $e(\mc) = (\nu-1)\ell(R/\mc)^2+\ell(R/\mc)$ is attained 
		if and only if $e(R)=(v-1)\ell(R/\mc)+1$ (cf. Theorem \ref{dimd}) and   $e(\mc)=e(R)\ell(R/\mc)$. 
		We observed in Theorem \ref{dimd} that the first equality holds if and only if $\ell(R/\mc)=1$; this fact is equivalent to $\mm=\mc$, that gives $e(\mc)=e(R)$; hence we get the second part of the thesis.
	\end{proof}
	
    Corollary \ref{pseudowilfring} extends \cite{DM}[Theorem 1], where a similar inequality was proved for numerical semigroup rings with combinatorial techniques. If $S$ is a numerical semigroup (i.e., a submonoid of $\mathbb{N}$ such that $\mathbb{N} \setminus S$ is finite) and $R=\Bbbk[[t^s \mid s \in S]]$ is a numerical semigroup ring, the invariants of $R$ appearing in this section can be read from combinatorial invariants of the numerical semigroup $S$. In fact, if we denote by $M= S \setminus \{0\}$, we have $e(R)= e(S) := \min M$, $\nu = \nu(S) := |M \setminus 2M|$, $e(\mc)= c(S) := \max ( \mathbb{Z} \setminus S) + 1$ and $\ell(R/\mc)=n(S) := |S \cap [0,c(S)-1]|$ (we refer to \cite{FGH} for a dictionary on the relation between invariants of numerical semigroups and their associated rings). In particular, the equality $e(\mc)=c(S)$ holds since, if $(y)$ is a minimal reduction of $\mc$, we can compute $e(\mc)$ as $\ell(R/(y))=\ell(\mc/y\mc)=\ell(\overline R/y\overline R)=\ell(\overline R/\mc)$. However, the following remark shows that these techniques can be interpreted by using Lech's inequality.
	
	\begin{remark}\label{depth}
		In \cite{DM}[Theorem 2] we have proved that $e(S) \le (v(S)-1)n(S)+1$ by studying the structure of the numerical semigroup, namely its Apéry set. Then we have considered the \emph{depth} of a numerical semigroup $q := \lceil \frac{c(S)}{e(S)} \rceil$, which is a combinatorial invariant that often appears in relation to Wilf's conjecture (cf. next section for its statement). Since, by construction, $\{0,e(S),\ldots,(q-1)\cdot e(S)\} \subseteq S \cap [0,c(S)-1]$, we have $q \le n(S)$, and thus, multiplying %\cite{DM}[Theorem 2]
		the above inequality by $q$, we get the desired result $$c(S) \le (\nu-1)n(S)^2+n(S).$$
		
		Here, Theorem \ref{dimd} achieves the same as \cite{DM}[Theorem 2] without relying on any combinatorial data. Moreover, since the invariants involved are all integers, the inequality $q = \lceil \frac{c(S)}{e(S) }\rceil\le n(S)$ is equivalent to $\frac{e(\mc)}{e(R)} \le \ell(R/\mc)$, which is exactly Lech's inequality in dimension one.  
	\end{remark}
	
	We notice that Lech's inequality for $\mc$ can be an equality also when $\mc\neq \mm$.
	In fact, if we restrict to numerical semigroup rings, Lech's inequality can be rephrased as $c(S) \le e(S) \cdot n(S)$, and equality is attained if and only if $n(S)=\frac{c(S)}{e(S)}$, that is, setting $k:=n(S)$, if and only if the associated numerical semigroup is of the form $S_k=\{0,e(S),\ldots,(k-1) \cdot e(S)\} \cup \{k \cdot e(S),\rightarrow\}$, thus providing examples of one-dimensional local Cohen-Macaulay rings, with arbitrarily large $\ell(R/\mc)$ attaining equality in Lech's inequality for the conductor.

	\begin{example}\label{dim1ex}
		In the regular local ring $K[[x,y,z]]$, consider the ideal $I=(x^2-y,x^3-z) \cap (x^2-z,x^3-y^2)$, and let $R=K[[x,y,z]]/I$. This ring can be seen as the image of the map $K[[x,y,z]] \rightarrow K[[t]] \times K[[u]]$ that sends $x \mapsto (t,u^2)$, $y \mapsto (t^2,u^3)$, $z \mapsto (t^3,u^4)$. This ring is associated to a singularity with two branches, and from its value semigroup in $\mathbb{N}^2$ we can deduce that $e(R)=3$, $e(\mc)=5$, $\nu=3$ and $\ell(R/\mc)=2$ (see \cite{BDF} for this kind of computations), and thus both the inequalities of Theorem \ref{dimd} and Corollary \ref{pseudowilfring} hold.
		The second one is very far from being sharp, and this happens because both Theorem \ref{dimd} and Lech's inequalities are not sharp for this ring.
	\end{example}

	\begin{discussion} The main assumptions made for $R$ in this paper are motivated as follows:
		we need the ring to be analytically unramified because in our argument we need that $\mc$ is a $\mm$-primary ideal (and the finiteness of $\ell(R/\mc)$), while the Cohen-Macaulayness is needed to compute the multiplicity of $R$ as a length as in \cite{HS}[Proposition 11.2.2]. Apparently, it might seem unclear whether the hypothesis that the ring has dimension one is needed, as one might think that our arguments could be adapted for higher dimension (considering a suitable minimal reduction of $\mm$);
        %in fact, there would be no contradiction in such a proof, 
        however, the set of rings with dimension $d > 1$ satisfying all the prescribed hypotheses is empty. This follows from \cite{G}[Proposition 5.21]: if $R$ is a reduced ring such that $\overline{R}$ is a finitely generated $R$-module, then the associated primes of the conductor $\mc$ have depth equal to one;
		therefore if $\mc$ is $\mm$-primary, $\mm$ must have depth equal to one, and thus for $d > 1$ the ring cannot be Cohen-Macaulay.
		
		It would be interesting to find a set of hypotheses that would allow to obtain upper bounds generalizing the one found in Theorem \ref{dimd} (and thus depending on $\nu-d$), for the multiplicity of rings with dimension $d > 1$. To this aim one should either propose a different argument not relying on the Cohen-Macaulay assumption or find a different quantity to substitute $\ell(R/\mc)$.
	\end{discussion}
	
	We conclude this section giving a class of analytically unramified, not Cohen Macaulay rings, with dimension $d > 1$, for which the conductor is $\mm$-primary.
	In a particular case we compute the invariants involved in our bound and we show that it still holds true (with $\nu-d$ in place of $\nu-1$).
	
	\begin{example}
		Let $(T, \mathfrak n)$ be a Cohen-Macaulay local ring which is integrally closed in its total ring of fractions $Q(R)$, and let $I$ be an $\mathfrak n$-primary (hence regular) ideal.  Let $\pi: T \rightarrow T/I$ be the natural surjection, let $B$ be a local Artinian subring of $T/I$ such that $B \subseteq T/I$ is finite. Let $R=\pi^{-1}(B)$. Hence we have the following pullback diagram:
		
		\begin{center}
			\begin{tikzcd}[row sep =1.5em, column sep =1em]
				R& & B \\
				T & & T/I \\
				\arrow[rightarrow, from=1-1, to=1-3,"\pi"]
				\arrow[rightarrow, from=2-1, to=2-3,"\pi"]
				\arrow[rightarrow, from=1-1,to=2-1,"\overline{i}"]
				\arrow[rightarrow,from=1-3,to=2-3,"i"]
			\end{tikzcd}
		\end{center}
		Notice that, from this diagram, $R$ can be described as $$R=\{(r,b) \ | \ \pi(r)=i(b)\} = \{(r,b) \ | \ r + I = b \} \subseteq T \times B.$$
		In light of the properties  of pullbacks proved in \cite{F}, we get that the following facts hold for $R$:
		\begin{enumerate}
			\item the homomorphism  $\overline i$
			is injective (\cite{F}[Corollary 1.5 (4)]), so $R$
			can be viewed as a subring of $T$;
			\item $I \subseteq Jac(T)=\mathfrak n$ and $B$ local imply that $R$ is a local ring (and we will denote its maximal ideal by $\mm$)
			(\cite{F}[Corollary 1.5 (1)]);
			\item since $I$ is a regular ideal, we get $Q(R)=Q(T)$ (\cite{F}[Corollary 1.5 (6)]);
			\item since $B \subset T/I$ is a finite extension, also $R\subseteq T$ is finite, thus integral (\cite{F}[Corollary 1.5 (4)]). So $T=\overline R$  is the integral closure of $R$ in $Q(R)$ and ${\rm dim} R = d$;
			\item $R$ is a Noetherian ring (\cite{F}[Proposition 1.8]);
			\item by construction, $I$ is an ideal of $R$ and $I \subseteq R : T$.  By $\dim T/I = 0$, it follows that $I$ is $\mm$-primary, and this implies that the conductor $R:T$ is also $\mm$-primary.
		\end{enumerate}
		
		The ring $R$ has dimension $d > 1$ and satisfies all the hypotheses outlined in the discussion: thus $R$ is not Cohen-Macaulay. 
        
		Notice that in general, the behavior of the Cohen-Macaulay property under pullbacks is not clearly understood; also, the invariants involved in this construction are difficult to compute. However, in some specific examples, we could verify that our inequality still holds, as we can see in the following particular situation.

		Let $T=K[[x,y]]$, $I=(x^2,y^2)$, $B=K[[x]]/(x^2)$. Thus we get the following diagram
		\begin{center}
			\begin{tikzcd}[row sep =1.5em, column sep =1em]
				R & & \frac{K[[x]]}{(x^2)} \\
				K[[x,y]] & & \frac{K[[x,y]]}{(x^2,y^2)} \\
				\arrow[rightarrow, from=1-1, to=1-3,"\pi"]
				\arrow[rightarrow, from=2-1, to=2-3,"\pi"]
				\arrow[rightarrow, from=1-1,to=2-1,"\overline{i}"]
				\arrow[rightarrow,from=1-3,to=2-3,"i"]
			\end{tikzcd}
		\end{center}
        from which we can deduce that $R=K[[x,x^2y,y^2,y^3]]$. While it is a two-dimensional, not Cohen-Macaulay ring, a direct computation with Macaulay2 (\cite{M2}) yields $e(R)=2$, $\nu=4$, $\ell(R/\mc)=2$, where $\mc=R:T=(x^2, x^2y, y^2,y^3)$ (as ideal of $R$): thus the inequality in Theorem \ref{dimd} holds (with $\nu-2$ in place of $\nu-1$).
	\end{example}
	
	\section{Wilf's conjecture for almost Gorenstein rings}
	
	For numerical semigroup rings the long-standing Wilf's Conjecture states that the bound obtained in \cite{DM}[Theorem 1] (that, more generally, coincides with the bound obtained in Corollary \ref{pseudowilfring}) can be divided by $\ell(R/\mc)$, namely that $e(\mc) \le v \cdot \ell(R/\mc)$. This conjecture was first proposed in 1978 (\cite{Wilf}), and seems very elusive. Many partial results were obtained by studying the combinatorics of numerical semigroups, but it is still unclear whether this inequality is specific to numerical semigroup rings, and thus dependent on some specific combinatorial property, or it could be extended to larger classes of rings through algebraic means. Our results, especially Remark \ref{depth}, suggest that the latter might be the case.
	
	In light of this, we ask whether the bound obtained in Corollary \ref{pseudowilfring} could be improved in a way that mirrors Wilf's Conjecture.
	\begin{que}\label{pseudoWilfConj} 
		Keeping the hypotheses of Corollary \ref{pseudowilfring}, is it true that $e(\mc) \le \nu \cdot \ell(R/\mc)$?
	\end{que}
	In this section, we affirmatively answer to the above question when the ring is almost Gorenstein.
	
	Almost Gorenstein rings are a class of Cohen-Macaulay rings introduced by Barucci and Fröberg in 1997 (cf. \cite{BF}), which are close to Gorenstein rings. This class of rings was originally defined in the context of analytically unramified rings of dimension one, though the definition was later generalized by Goto, Matsuoka and Phuong to include the ramified case in dimension one (cf. \cite{GM}) and by Goto, Takahashi and Taniguchi (cf. \cite{GTT}) for the higher dimensional case. Since in Theorem \ref{dimd} we consider analytically unramified rings in dimension one, we will follow the original definition.
	
	We recall that, under the assumption of the present paper, $R$ always has a canonical module $\omega$ which is isomorphic to a fractional ideal (see \cite{HK}[Satz 6.21]), so it is usually called canonical ideal (saying fractional ideal we always assume that it contains a non-zero divisor of $Q(R)$). As fractional ideal,
	a canonical module is characterized by the property that, for any other fractional ideal $J$, the equality $J=\omega : (\omega : J)$ holds.  Moreover, any two canonical ideals $\omega, \omega'$ are isomorphic, i.e. in our hypotheses, there is a non-zero divisor $z$ of $Q(R)$ such that $\omega=z\omega'$. This fact implies that we can always choose 
	$\omega$ with the property
	that $R \subseteq \omega \subseteq \overline{R}$. We also recall that if $I \subseteq J$ are fractional ideals, then 
	$\ell(I/J)=\ell((\omega:J)/(\omega:I))$ (cf. \cite{HK}[Bermerkung 2.5]). Moreover, with this choice of $\omega$ we have $\omega:\overline R=\omega:\omega\overline R=(\omega:\omega)\overline R=R:\overline R=\mc$.

	%If $R$ is a one-dimensional Cohen-Macaulay ring, a fractional ideal $\omega$ is called a \emph{canonical ideal} of $R$ if $\omega$ contains a non-zero divisor and if for any fractional ideal $J$ which contains a non-zero divisor we have $J=\omega : (\omega : J)$. In particular, we have $R= \omega : \omega$, and $\ell(I/J)=\ell((\omega:J)/(\omega:I))$ if $J \subseteq I$ are fractional ideals (cf. \cite{BF}). If $R$ is analytically unramified, then $R$ always admits a canonical ideal. Moreover, according to \cite[Corollary 18]{BF}, there is always a canonical ideal $\omega$ such that $R \subseteq \omega \subseteq \overline{R}$ (this ideal is sometimes called the \emph{canonical module}).
		
	Now we can give the definition of almost Gorenstein rings in our setting: a one-dimensional, analytically unramified, Cohen-Macaulay local ring with infinite residue field,  is \emph{almost Gorenstein} if $\omega \subseteq \mm : \mm$, where $\omega$ is any canonical ideal $R\subseteq \omega \subseteq \overline R$ (by \cite{BF}[Definition-Proposition 20] the definition is independent of the choice of $\omega$). 
    %This definition has been later generalized by Goto et al. in the context of one-dimensional rings (cf. \cite{GM}), and then extended to the higher dimension case (cf. \cite{GTT}). 
	
	Numerical semigroups associated to almost Gorenstein numerical semigroup rings are called \emph{almost symmetric}. To give this definition, we need to define canonical ideals in the context of numerical semigroups (we refer to \cite{ADG} for a detailed treatment of canonical ideals in this context). Let $S$ be a numerical semigroup. A \emph{relative ideal} $I$ of $S$ is a subset of $\mathbb{Z}$ such that $I + S \subseteq I$ and $s+I = \{s+i \ | \ i \in I\} \subseteq S$ for some $s \in S$. An \emph{ideal} of $S$ is a relative ideal $I$ such that $I \subseteq S$. The ideal $M_S= S \setminus \{0\}$ is called the \emph{maximal ideal} of $S$. For a numerical semigroup $S$, denote by $F(S) := \max \mathbb{Z} \setminus S = c(S)-1$ the \emph{Frobenius number} of $S$. The set $$\omega_S := \{F(S)- \alpha \ | \ \alpha \not \in S\} \subseteq \mathbb{N}$$
		is a relative ideal of $S$ such that $S \subseteq \omega_S$, which is called the (standard) \emph{canonical ideal} of $S$.
        %We say that an ideal $I$ of $S$ is \emph{canonical} if $I=z + \omega_S$ for some $z \in \mathbb{Z}$.
	
	Let $R$ be the numerical semigroup ring $R = \Bbbk[[t^s \ | \ s \in S]]$ associated to a numerical semigroup $S$. Then a fractional ideal $\omega$ such that $R \subseteq \omega \subseteq \overline{R}$ is a canonical ideal of $R$ if and only if its value set $v(\omega)$ is the standard canonical ideal of $S$ (see \cite[Satz 5]{J}), whence $\omega \cong \Bbbk[[t^{F(S)-\alpha} \ | \ \alpha \not \in S]].$
    
    A numerical semigroup $S$ with (unique) minimal generating set $\{n_1,\ldots,n_\nu\}$ is \emph{almost symmetric} if $\omega_S \subseteq M_S - M_S := \{z \in \mathbb{Z} \ | \ z + M_S \subseteq M_S\}$, or equivalently if $F(S)-\alpha+n_i \in S$ for every $\alpha \not \in S$ and $i=1,\ldots,\nu$. These semigroups were originally introduced alongside almost Gorenstein rings in Barucci and Fröberg's work, but they are also studied on their own. In particular, they enjoy nice symmetrical properties (cf. \cite{N}). Moreover, Wilf's conjecture is known to be true for this class of numerical semigroups, but all the proofs of this fact rely on numerical arguments and cannot be adapted for rings (see \cite{B} and \cite{DM}).
	
	\begin{thm}\label{AGWilf}
		Let $(R,\mm)$ be a one-dimensional, analytically unramified, Cohen-Macaulay local ring with infinite residue field. If $R$ is almost Gorenstein, then $e(\mc) \le \nu \cdot \ell(R/\mc)$.
	\end{thm}
	\begin{proof}
		Let $(x)$ be a minimal reduction of $\mm$. We have that $e(\mc)=\ell(\overline R/\mc)=\ell(\overline R/x\overline R)+\ell(x\overline R/\mm)+\ell(\mm/\mc)=e(R)+\ell(x\overline R/\mm)+\ell(\mm/\mc)$. 
		Keeping in mind the proof of Theorem \ref{dimd}, we also have $\ell(\mm/x\mc)=e(R)+\ell(\mm/\mc)$.
		Hence $e(\mc)=\ell(x\overline R/\mm)+\ell(\mm/x\mc)$.
		Moreover, we have the surjective homomorphism of $R$-modules
		$\varphi:(R/\mc)^{\nu} \rightarrow \mm/x\mc$, that gives 
		$\nu\cdot \ell(R/\mc)=\ell(\mm/x\mc)+
		\ell(\ker\varphi)$. Therefore, 
		if we prove that $\ell(x\overline R/\mm)\leq \ell(\ker\varphi)$, we obtain the thesis.
		
		We observe that $\ell(x\overline R/\mm)=\ell((\omega:\mm)/(\omega:x\overline R))$. Now $\omega:x\overline R=x^{-1}(\omega:\overline R)=x^{-1}\mc$ and, remembering that $\mm=x\overline R\cap R$ (cf. Remark \ref{m}), $\omega:\mm=
		\omega:(x\overline R\cap R)=(\omega:x\overline R)+(\omega:R)=x^{-1}\mc+\omega$. Therefore, $\ell((\omega:\mm)/(\omega:x\overline R))=\ell((x^{-1}\mc+\omega)/x^{-1}\mc)=\ell(\omega/(x^{-1}\mc\cap\omega))$.
		
		Hence, to prove the thesis, it is enough to show that we can embed 
		$\omega/(x^{-1}\mc\cap\omega)$ in $\ker\varphi$.
		
		If $\{x,x_2,\dots,x_{\nu}\}$ is a minimal generating system of $\mm$,
		we can fix any $x_i\neq x$ (for simplicity of notation we choose i=2) and define $\psi:\omega/(x^{-1}\mc\cap\omega) \rightarrow (R/\mc)^{\nu}$, setting $\psi(\overline{\alpha})=(-\overline{\alpha x_2}, \overline{\alpha x}, \overline 0, \dots, \overline 0)$ (where the symbol $\overline{()}$ has the obvious different meanings). Notice that $\psi$ is a well-defined 
		homomorphism of $R$-modules; in fact, if $\alpha-\beta \in x^{-1}\mc$, both $(\alpha-\beta)x$ and $(\alpha-\beta)x_2=(\alpha-\beta)xx_2x^{-1}$
		belong to $\mc$, since $x_2x^{-1}\in \overline R$. Moreover, if $\overline{\alpha} \neq \overline 0$ in $\omega/(x^{-1}\mc\cap\omega)$, and so $\alpha \notin x^{-1}\mc$, then $\overline{\alpha x} \neq\overline 0$ in $R/\mc$; hence $\psi$ is injective. As for 
		$\alpha x_2$, it could belong or not belong to $\mc$; in the first case the image of $\overline{\alpha}$ has zero first component.
		
		Consider now $\varphi((-\overline{\alpha x_2}, \overline{\alpha x}, \overline 0, \dots, \overline 0))=\overline{-\alpha x_2x+\alpha x x_2}$. 
		If $\alpha x_2 \notin \mc$, then $\alpha x_2 x \notin x\mc$, and  
		$\overline{-\alpha x_2x+\alpha x x_2}=\overline 0$; on the other hand,
		if $\alpha x_2 \in \mc$ we have
		$\psi(\overline{\alpha})=(\overline{0}, \overline{\alpha x}, \overline 0, \dots, \overline 0)$, and applying $\varphi$ we get $\overline{\alpha xx_2}=\overline 0$,
		since $\alpha xx_2 \in x\mc$. 
		The proof is complete.
	\end{proof}
	
	%Numerical semigroups associated to almost Gorenstein numerical semigroup rings are called \emph{almost symmetric}. These semigroups were originally introduced alongside almost Gorenstein rings in Barucci and Fröberg's work, but they are also studied on their own. In particular, they enjoy nice symmetrical properties (cf. \cite{N}). Moreover, Wilf's conjecture is known to be true for this class of numerical semigroups (see for example \cite{DM} for a proof).

	\begin{example}
		\begin{enumerate}
			\item Let $R=\Bbbk[[t^7,t^9,t^{11},t^{19}]]$ be the numerical semigroup ring associated to the numerical semigroup $S= \langle 7,9,11,19 \rangle$. $R$ is almost Gorenstein since $S$ is an almost symmetric numerical semigroup. From the correspondence between numerical semigroup rings and their value semigroup, we can compute $e(R)=7$, $\ell(R/\mc)=12$, $\nu = 4$, $e(\mc)=25$ (in particular the inequality $e(\mc) \le v \cdot \ell(R/\mc)$ holds). From the proof of Theorem \ref{AGWilf} we can see that $\ell(\mm / x \mc) = e(R)+\ell(\mm/\mc)=7+11=18,$ and thus $\ell(\ker \varphi)=30$. The standard canonical ideal of $S$ is $\omega_S=S \cup (12 + S)$, therefore a canonical ideal of $R$ is $\omega= R + t^{12}R$. A minimal reduction of $R$ is $x=t^7$, which is part of the minimal generating system $\{t^7,t^9,t^{11},t^{19}\}$ of $\mm$. The map $\psi: \omega/(x^{-1}\mc \cap \omega) \rightarrow (R/\mc)^\nu$ defined by considering $x_2=t^9$ is such that $\psi(\overline{t^{12}})=(-\overline{t^{21}},\overline{t^{19}},\overline{0},\overline{0})$. On the other hand, choosing $x_2=t^{19}$ yields the map $\psi$ such that $\psi(\overline{t^{12}})=(-\overline{t^{29}}=\overline{0},\overline{t^{19}},\overline{0},\overline{0})$.
            In the first case, when we apply the homomorphism $\varphi$ of the previous theorem we get $\overline{-t^{28}+t^{28}}=\overline 0$; in the second case we get $\overline{t^{38}}=\overline0$, since $t^{38} \in t^7\mc$.
			\item Let $R=K[[x,y,z,w]]/((x^4-y^3,z,w)\cap(x,y,z^4-w^3))$, which can be seen as the image of the map
			$K[[x,y,z,w]] \rightarrow K[[t]]\times K[[u]]$ defined sending $x \mapsto (t^3,0)$, $y\mapsto (t^4,0)$, $z\mapsto (0,u^3)$ and $w\mapsto (0,u^4)$. To this ring we can associate a good subsemigroup of $\mathbb N^2$ (see \cite{BDF}); this semigroup is depicted with black
			dots in the figure below and, following the computation of lengths explained in \cite{BDF}, it can be used to  determine the invariants involved in our results.  More precisely, we have that $e(R)=6$, $\nu=4$, $e(\mc)=12$ and $\ell(R/\mc) =5$. 
			Moreover, $\ell(\mm/x\mc)=e(R)+\ell(\mm/\mc)=6+4=10$; so $\ell(\ker \varphi)=10$.
			
			Also canonical ideals between $R$ and $\overline R$ can be characterized by their value set, which is depicted with black dots  and with circles in the same figure. In particular, a canonical ideal between $R$ and $\overline R$ is $R+(t^3,1)R+(t^5,u^5)R$. Again using the computation of lengths explained in \cite{BDF}, one can easily check that $\ell(x\overline R/\mm)=\ell(\omega/(x^{-1}\mc \cap \omega))=2$.
			\begin{center}
				\begin{tikzpicture}[scale=0.5]
					\coordinate (Origin)   at (0,0);
					\coordinate[label=below:0] (zero) at (-0.6,-0.2);
					\coordinate (XMin) at (-1,0);
					\coordinate (XMax) at (11,0);
					\coordinate (YMin) at (0,-1);
					\coordinate (YMax) at (0,11);
					\draw [-latex] (XMin) -- (XMax);
					\draw [-latex] (YMin) -- (YMax);
					
					\node[draw,circle,inner sep=1.2pt,fill] at (0,0) {};
					\node[draw,circle,inner sep=1.2pt,fill] at (3,3) {};
					\node[draw,circle,inner sep=1.2pt,fill] at (3,4) {};
					\node[draw,circle,inner sep=1.2pt,fill] at (4,3) {};
					\node[draw,circle,inner sep=1.2pt,fill] at (4,4) {};
					
					\foreach \x in {1,...,9}{
						\coordinate[label=below:\x] (\x) at (\x,-0.2);
						\coordinate[label=left:\x] (\x') at (-0.2,\x);
						\coordinate (start\x) at (\x,-0.1);
						\coordinate ('start\x) at (-0.1,\x);
						\coordinate (end\x) at (\x,10);
						\coordinate ('end\x) at (10,\x);
					}
					
					\foreach \x in {6,...,9}{
						\foreach \y in {6,...,9}{
							\node[draw,circle,inner sep=1.2pt,fill] at (\x,\y) {};
						}
					}
					
					\foreach \x in {3,4,5}{
						\node[draw,circle,inner sep=1.2pt,fill=white] at (\x,5) {};
						\node[draw,circle,inner sep=1.2pt,fill=white] at (5,\x) {};
					}
					
					\foreach \x in {3,4}{
						\foreach \y in {6,...,9}{
							\node[draw,circle,inner sep=1.2pt,fill] at (\x,\y) {};
							\node[draw,circle,inner sep=1.2pt,fill] at (\y,\x) {};
						}
					}
					
					\foreach \x in {3,...,9}{
						\node[draw,circle,inner sep=1.2pt,fill=white] at (\x,0) {};
						\node[draw,circle,inner sep=1.2pt,fill=white] at (0,\x) {};
					}
					
					%\begin{scope}[on background layer]  
					\foreach \x in {1,...,9}{
						\draw [thin, gray] (start\x) -- (end\x);
						\draw [thin, gray] ('start\x) -- ('end\x);}
					%\end{scope}
				\end{tikzpicture}
			\end{center}     
			This ring is almost Gorenstein; in fact, 
			it is not difficult to see that the classes of $x,y,z,w$ belong to $\mm:\omega$. To follow the proof of the Theorem \ref{AGWilf} we need a minimal reduction of $\mm$, which is e.g. $\overline{x+y}$, that corresponds to $(t^3,u^3)$;
			we have $\mm=(\overline{x+y}, \overline x, \overline z, \overline w)$.
			Choose any element in $\omega$, not belonging to $(t^3,u^3)^{-1}\mc$, i.e. with value on the axes, for example $(t^3,1)$; we get $\varphi(\psi((\overline{t^3,1})))=\varphi((\overline{t^6,0}),(\overline{t^6,u^3}),\overline 0, \overline 0)=\overline{(t^9,0)-(t^9,0)}=\overline 0$. 
		\end{enumerate}
	\end{example}
	
	To conclude the paper we consider a generalization of almost Gorenstein rings, i.e. rings with canonical reduction, introduced and studied in \cite{R}. A ring $(R,\mm)$, satisfying the hypotheses of this paper, has \emph{canonical reduction} if, by definition, there exists a canonical ideal that is a reduction of $\mm$. If we consider $\omega$ such that $R \subseteq \omega \subseteq \overline R$, this is equivalent to say that there exists a minimal reduction $(x)$ of $\mm$ such that $x\omega \subseteq \mm$. In  \cite[Theorem 3.13]{R} it is proved that a numerical semigroup ring $R$ has a canonical reduction if and only if the associated numerical semigroup $S$ satisfies the hypothesis $F(S)-\alpha+e(S) \in S$ for every $\alpha \not \in S$ or, equivalently, that $\omega_S \subseteq M_S - e(S)$. Numerical semigroups satisfying this property were already introduced independently in \cite{BFR} under the name of \emph{positioned} numerical semigroups. Notice that since $e(S)$ is a minimal generator of $S$, then almost symmetric numerical semigroups are positioned, reflecting the relation between almost Gorenstein rings and rings with canonical reduction. Moreover, in \cite{BFR} it has been proved that positioned numerical semigroups satisfy Wilf's conjecture, but, again, the proof relies on numerical arguments. Repeating the proof of Theorem \ref{AGWilf}, we obtain that rings with canonical reduction satisfy the inequality $e(R)\leq \nu\cdot \ell(R/\mc)$, provided that they fulfill an extra condition.
	
	\begin{cor}\label{AGCR}
		Let $(R,\mm)$ be a one-dimensional, analytically unramified, Cohen-Macaulay local ring with infinite residue field. If $R$ has a canonical reduction $x\omega$ (with $(x)$ minimal reduction of $\mm$) and if there exists $x_i$ belonging to a minimal generating set $\{x,x_2,\dots,x_{\nu}\}$ of $\mm$, such that $x_i\in \mm:\omega$, then $e(\mc) \le \nu \cdot \ell(R/\mc)$.
	\end{cor}
	
	\begin{example}
		\begin{enumerate}
			\item Let $R=\Bbbk[[t^7,t^9,t^{11},t^{13}]]$ be the numerical semigroup ring associated to the numerical semigroup $S=\langle 7,9,11,13 \rangle$. The semigroup $S$ is not almost symmetric, hence $R$ is not almost Gorenstein; however, $S$ is positioned, and thus $R$ is a ring with canonical reduction. A minimal generating set of $\mathfrak{m}$ is $\{t^7,t^9,t^{11},t^{13}\}$, where $(t^7)$ is a minimal reduction of $\mathfrak{m}$ and $t^9 \in \mathfrak{m} : \omega$. Therefore, the hypotheses of Corollary \ref{AGCR} are satisfied. In fact, from the correspondence between numerical semigroup rings and their value semigroups we obtain $\ell(R/\mc)=8$, $v=4$, $e(\mc)=20$, and thus $e(\mc) \le \nu \cdot \ell(R/\mc)$.
			\item Let $R=\Bbbk[[t^{17},t^{27},t^{29}]]$ be the numerical semigroup ring associated to $S=\langle 17,27,29 \rangle$. $R$ is a ring with canonical reduction, not almost Gorenstein. A minimal generating set of $\mathfrak{m}$ is $\{t^{17},t^{27},t^{29}\}$, where $t^{17}$ is a minimal reduction of $\mathfrak{m}.$ However, $t^{27},t^{29} \not \in \mathfrak{m} : \omega$, thus the hypotheses of Corollary \ref{AGCR} are not satisfied. Nonetheless, since $S$ is a positioned numerical semigroup, Wilf's conjecture holds for $S$ (and thus for $R$), and in fact by computing the relevant invariants from the value semigroup, i.e.  $\ell(R/\mathfrak{c})=74$, $v=3$, $e(\mathfrak{c})=158$, we see that $e(\mc) \le \nu \cdot \ell(R/\mc)$.
			\item Let $R=\Bbbk[[t^7,t^9,t^{11},t^{15}]]$ be the numerical semigroup ring associated to the numerical semigroup $S=\langle 7,9,11,15 \rangle$. This numerical semigroup is not positioned, hence $R$ is not a ring with canonical reduction - and thus it does not satisfy the hypotheses of Corollary \ref{AGCR}. However, computing the invariants $e(\mc),\ell(R/\mc)$ from value semigroup we notice that these invariants are the same as in the first example, that is, $\ell(R/\mc)=8$, $v=4$, $e(\mc)=20$. Thus, even if $R$ is not a ring with canonical reduction, it still verifies the inequality $e(\mc) \le \nu \cdot \ell(R/\mc)$.
		\end{enumerate}
		
	\end{example}

		\section*{Acknowledgements}
		We would like to thank Marilina Rossi for helpful suggestions and for directing us to Lech's inequality, which played a crucial role in proving one of the main results.
		
		We also thank Carmelo Antonio Finocchiaro for his help in building the framework for the example in higher dimension.
		
		Finally, we thank the anonymous referee of a first version of this paper for his remarks and suggestions, which led to a significant improvement of the proof of Theorem \ref{dimd}.
		
		Both authors were partially funded by the project “Proprietà locali e globali di anelli e
		di varietà algebriche”-PIACERI 2020-22, Università degli Studi di Catania.
		The first author was also supported by
		the PRIN 2020 "Squarefree Gröbner degenerations, special
		varieties and related topics". The second author was partially supported by the Austrian Science Fund FWF, grant DOI 10.55776/PAT9756623.

	\end{document}